\documentclass[final,12pt]{elsarticle}
\usepackage{srcltx}
\usepackage{eurosym}
\usepackage{mathtools}
\usepackage{amsmath}
\usepackage{amsfonts}
\usepackage{amssymb}
\usepackage{amsthm}
\usepackage{graphicx}
\usepackage{mathrsfs}
\usepackage{xcolor}
\usepackage{exscale}
\usepackage{latexsym}
\usepackage{showkeys}

\numberwithin{equation}{section}

\usepackage[colorlinks,plainpages=true,pdfpagelabels,hypertexnames=true,colorlinks=true,pdfstartview=FitV,linkcolor=blue,citecolor=red,urlcolor=black]{hyperref}
\PassOptionsToPackage{unicode}{hyperref}
\PassOptionsToPackage{naturalnames}{hyperref}
\usepackage{enumerate}
\usepackage[shortlabels]{enumitem}
\usepackage{bookmark}
\usepackage{wasysym}
\usepackage{esint}
\usepackage[ddmmyyyy]{datetime}
\usepackage[margin=2cm]{geometry}
\makeatletter
\g@addto@macro\normalsize{%
	\setlength\abovedisplayskip{4pt}
	\setlength\belowdisplayskip{4pt}
	\setlength\abovedisplayshortskip{4pt}
	\setlength\belowdisplayshortskip{4pt}
}
\everymath{\displaystyle}
\usepackage[capitalize,nameinlink]{cleveref}
\crefname{section}{Section}{Sections}
\crefname{subsection}{Subsection}{Subsections}
\crefname{condition}{Condition}{Conditions}
\crefname{hypothesis}{Hypothesis}{Conditions}
\crefname{assumption}{Assumption}{Assumptions}
\crefname{lemma}{Lemma}{Lemmas}
\crefname{definition}{Definition}{Definitions}

\crefformat{equation}{\textup{#2(#1)#3}}
\crefrangeformat{equation}{\textup{#3(#1)#4--#5(#2)#6}}
\crefmultiformat{equation}{\textup{#2(#1)#3}}{ and \textup{#2(#1)#3}}
{, \textup{#2(#1)#3}}{, and \textup{#2(#1)#3}}
\crefrangemultiformat{equation}{\textup{#3(#1)#4--#5(#2)#6}}%
{ and \textup{#3(#1)#4--#5(#2)#6}}{, \textup{#3(#1)#4--#5(#2)#6}}%
{, and \textup{#3(#1)#4--#5(#2)#6}}

\Crefformat{equation}{#2Equation~\textup{(#1)}#3}
\Crefrangeformat{equation}{Equations~\textup{#3(#1)#4--#5(#2)#6}}
\Crefmultiformat{equation}{Equations~\textup{#2(#1)#3}}{ and \textup{#2(#1)#3}}
{, \textup{#2(#1)#3}}{, and \textup{#2(#1)#3}}
\Crefrangemultiformat{equation}{Equations~\textup{#3(#1)#4--#5(#2)#6}}%
{ and \textup{#3(#1)#4--#5(#2)#6}}{, \textup{#3(#1)#4--#5(#2)#6}}%
{, and \textup{#3(#1)#4--#5(#2)#6}}

\crefdefaultlabelformat{#2\textup{#1}#3}

\newtheorem{theorem} {Theorem}[section]
\newtheorem{proposition}[theorem]{Proposition}
\newtheorem{lemma}[theorem]{Lemma}

\newtheorem{counter example}[theorem]{Counter Example}
\newtheorem{remark}[theorem] {Remark}
\newtheorem{definition}[theorem] {Definition}

\def\CC{{\rm \kern.24em \vrule width.02em height1.4ex depth-.05ex \kern-.26emC}}

\def\TagOnRight

\def\AA{{it I} \hskip-3pt{\tt A}}

\def\QQ{\rlap {\raise 0.4ex \hbox{$\scriptscriptstyle |$}} {\hskip -0.1em Q}}

\makeatletter
\newcommand{\vo}{\vec{o}\@ifnextchar{^}{\,}{}}
\makeatother

\def\YYint#1#2#3{{\setbox0=\hbox{$#1{#2#3}{\iint}$}
		\vcenter{\hbox{$#2#3$}}\kern-.50\wd0}}

\def\XXint#1#2#3{{\setbox0=\hbox{$#1{#2#3}{\int}$}
		\vcenter{\hbox{$#2#3$}}\kern-.50\wd0}}

\makeatletter
\def\namedlabel#1#2{\begingroup
	\def\@currentlabel{#2}%
	\label{#1}\endgroup
}
\makeatother
\makeatletter
\newcommand{\rmh}[1]{\mathpalette{\raisem@th{#1}}}
\newcommand{\raisem@th}[3]{\hspace*{-1pt}\raisebox{#1}{$#2#3$}}
\makeatother

\newcounter{desccount}

\newcommand{\descref}[2]{\hyperref[#1]{\textcolor{black}{}\textcolor{blue}{ #2}\textcolor{black}{}}}

\newcommand{\dref}[2]{\hyperref[#1]{\textcolor{black}{(}\textcolor{blue}{\bf #2}\textcolor{black}{)}}}
\newcommand{\pa} {\partial}

\newcommand{\al} {\alpha}
\newcommand{\rr}{\rightarrow}

\newcommand{\B} {\beta}
\newcommand{\de} {\delta}

\newcommand{\p}  {\prime}
\newcommand{\e}  {\epsilon}

\newcommand{\la} {\lambda}

\newcommand{\f}{\infty}

\newcommand{\noi} {\noindent}

\DeclareMathOperator{\dv}{div}

\newcommand{\norm}[1]{\left|\hspace{-0.2mm}\left| #1 \right|\hspace{-0.2mm}\right|}
\newcommand{\abs}[1]{\left| #1\right|}

\newcounter{whitney}
\refstepcounter{whitney}

\newcounter{ineqcounter}
\refstepcounter{ineqcounter}
\makeatletter
\def\ps@pprintTitle{%
	\let\@oddhead\@empty
	\let\@evenhead\@empty
	\def\@oddfoot{}%
	\let\@evenfoot\@oddfoot}
\makeatother
\usepackage[doublespacing]{setspace}
\usepackage[titletoc,toc,page]{appendix}

\makeatletter
\newcommand{\refcheckize}[1]{%
	\expandafter\let\csname @@\string#1\endcsname#1%
	\expandafter\DeclareRobustCommand\csname relax\string#1\endcsname[1]{%
		\csname @@\string#1\endcsname{##1}\wrtusdrf{##1}}%
	\expandafter\let\expandafter#1\csname relax\string#1\endcsname
}
\makeatother

\refcheckize{\cref}
\refcheckize{\Cref}


\makeatletter
\newcommand{\mainsectionstyle}{%
	\renewcommand{\@secnumfont}{\bfseries}
	\renewcommand\section{\@startsection{section}{2}%
		\z@{.5\linespacing\@plus.7\linespacing}{-.5em}%
		{\normalfont\bfseries}}%
}
\makeatother
\usepackage{pgf,tikz}
\usetikzlibrary{arrows}
\usetikzlibrary{decorations.pathreplacing}

\usepackage{xpatch}
\xpatchcmd{\MaketitleBox}{\hrule}{}{}{}
\xpatchcmd{\MaketitleBox}{\hrule}{}{}{}

\date{}
\newcommand{\R}{\mathbb{R}}
\newcommand{\T}{\mathbb{T}}

\newcommand{\G}{\mathbb{G}}
\newcommand{\Pp}{\mathbb{P}}

\newcommand{\F}{\mathfrak{F}}

\newcommand{\E}{\mathbb{E}}

\newcommand{\U}{\mathcal{U}}

\newcommand{\rd}{\mathrm{d}}
\newcommand{\dom}{\mathbb{T}^d}
\newcommand{\bvr}{\pmb{\varphi}}
\newcommand{\nbf}{\textbf{n}}
\newcommand{\mbf}{\textbf{m}}

\begin{document}
	%
	
	\begin{frontmatter}
		\title{Uniqueness and energy balance for isentropic Euler equation with stochastic forcing}
		
		\author[myaddress1]{Shyam Sundar Ghoshal}\ead{ghoshal@tifrbng.res.in}
		\author[myaddress1]{Animesh Jana}\ead{animesh@tifrbng.res.in}
		\author[myaddress1]{Barun Sarkar}\ead{barunsarkar.math@gmail.com}
		\address[myaddress1]	{Tata Institute of Fundamental Research,Centre For Applicable Mathematics,
			Sharada Nagar, Chikkabommsandra, Bangalore 560065, India.}
		
		\begin{abstract}
			In this article, we prove uniqueness and energy balance for isentropic Euler system driven by a cylindrical Wiener process. Pathwise uniqueness result is obtained for weak solutions having H\"older regularity $C^{\al},\al>1/2$ in space and satisfying one-sided Lipschitz bound on velocity. We prove Onsager's conjecture for isentropic Euler system with stochastic forcing, that is, energy balance equation for solutions enjoying H\"older regularity $C^{\al},\al>1/3$. Both the results have been obtained in a more general settings by considering regularity in Besov space.
		\end{abstract}
	
	\end{frontmatter}

\tableofcontents	
\section{Introduction}
This article deals with the uniqueness as well as the energy-balance equation for the following compressible Euler system driven by cylindrical
Wiener process,
\begin{align}
\rd\varrho+\dv_x\mbf \rd t&=0,\label{SCE1}\\
\rd\mbf+\dv_x\left(\frac{\mbf\otimes\mbf}{\varrho}\right)\rd t+\nabla_xp(\varrho)\rd t&=\mathbb{G}(\varrho,\mbf)\rd W_t,\label{SCE2}\\
\varrho(0,\cdot)=\varrho_0,&\ \ \mbf(0,\cdot)=\mbf_0,\nonumber
\end{align}
where $\varrho=\varrho(t,x)$, $\mbf=\mbf(t,x)$ denote the density and momentum of a compressible fluid respectively, the barotropic pressure is $p=p(\varrho)=\kappa\varrho^\gamma$ for $\kappa>0,\gamma>1$. We work in $\T^d$ for space variable to avoid problems related to the presence of kinematic boundary, where $\T^d$ denotes a flat torus:
\[\T^d\equiv \left( [-1,1]\big|_{\{-1,1\}}\right)^d;\ \text{for}\ d=2,3.  \]
In the above system \eqref{SCE1}--\eqref{SCE2} the force term is a cylindrical Wiener process $\{W_t\}_{t\geq0}$ on a filtered probability space $(\Omega,\F, \{\F_t\}_{t\geq0},\Pp)$ 
and the diffusion coefficient $\mathbb{G}(\varrho,\mbf)$ takes values in $L_2(\U;L^2(\T^d))$, the space of Hilbert-Schmidt operators. In this article, we consider solutions having Besov regularity, which is defined as follows: 
 let $\al\in(0,1)$, $1\leq q<\f$ and $\mathcal{O}\subset\bar{\mathcal{O}}\subset\dom$ be open bounded set. Then, $B^{\al,\f}_q(\mathcal{O},\R_+\times\R^d)$ is the set of all $f=(f_1,f_2)\in L^q(\mathcal{O},\R_+)\times L^q(\mathcal{O},\R^d)$ functions such that $\abs{f}_{B^{\al,\f}_q(\mathcal{O},\R_+\times\R^d)}<\f$ where 
\begin{equation}
\abs{f}_{B^{\al,\f}_q(\mathcal{O},\R_+\times\R^d)}:=\sup\limits_{0\neq\xi\in\R^d,\xi+\mathcal{O}\subset\dom}\abs{\xi}^{-\al}\norm{f(\cdot+\xi)-f(\cdot)}_{L^q(\mathcal{O},\R_+\times\R^d)}.
\end{equation}

One of the main goal of this article is to obtain the pathwise uniqueness for weak solutions of \eqref{SCE1}--\eqref{SCE2} with $B^{\al,\f}_q,\al>1/2$ regularity in space satisfying the following one-sided Lipschitz condition on velocity, that is,
\begin{equation}\label{ineq:OSC}
\nabla_x\left(\frac{\mbf}{\varrho}\right):(\xi\otimes\xi)+\chi(t)\abs{\xi}^2\geq0\mbox{ in }\mathcal{D}^{\p}(\dom)\mbox{ for all }\xi\in\R^d,
\end{equation}
where $\varrho>0$, $\chi \in L^1(0,T)$ and $\mathcal{D}^{\p}$ denotes the space of distributions. In the deterministic setting, the condition \eqref{ineq:OSC} has been introduced in \cite{fgj} for isentropic Euler system in order to prove uniqueness with low regularity. We note that planar rarefaction waves for deterministic isentropic Euler system satisfies \eqref{ineq:OSC} and the mentioned Besov regularity as well. Novelty of the present work is: we do not need to assume extra regularity in time and the Besov regularity in space is required only for $t>0$ which allows the initial data to be discontinuous.  It is worth mentioning that strong solutions of \eqref{SCE1}--\eqref{SCE2} satisfy the Besov regularity as mentioned above and the condition \eqref{ineq:OSC}. Hence, pathwise weak-strong uniqueness for \eqref{SCE1}--\eqref{SCE2} follows from our uniqueness result. Another main objective of our present work is to show the energy-balance equation for weak solutions to the system \eqref{SCE1}--\eqref{SCE2}. This is famously known as Onsager's conjecture for incompressible Euler system.

In the literature, there are few results known for hyperbolic systems with stochastic perturbation. Existence of weak martingale solutions has been established  \cite{bv} for stochastic isentropic Euler equation with an adaptation of kinetic formulation. Ill-posedness of compressible Euler system with stochastic forcing has been proved \cite{bfh-ilpos} by an appropriate adaptation of convex integration to stochastic setting.  In multi dimension, existence of strong solution is proved \cite{kim} for hyperbolic system of conservation laws with stochastic forcing. For isentropic Euler system in multi dimension with stochastic forcing, existence of strong solution is established in \cite{bm}. 
 Note that strong solutions obtained in \cite{bm} are included in our setting, that is, in the class of Besov solutions with one-sided Lipschitz condition.  

For compressible Navier-Stokes equation with stochastic forcing, existences of martingale solutions and local strong solution are known due to \cite{Breit-Hof} and \cite{bfh} respectively. Pathwise uniqueness via relative entropy method has been established \cite{bfh-ren} for compressible Navier-Stokes equation driven by cylindrical Wiener process. We refer to \cite{BrFeiHofMa} for study of stationary solutions compressible viscous flow driven by stochastic perturbation.

In this article, we prove the uniqueness result by a suitable adaptation of the relative energy method to stochastic setting. Note that in a typical weak-strong uniqueness proof (in deterministic setting), one uses the strong solution as a test function in the weak-formulation. Here we only assume $1/2$--Besov regularity in space, this does not allow us to use it as a test function in the weak formulation. One idea could be mollifying the system and then pass to the limit.  
In stochastic setting we can not do a time mollification. We resolve this issue by considering the system in $(\varrho,\mbf)$ variable instead of $(\varrho,\textbf{u})$ variable (here $\textbf{u}$ stands for velocity, it can be written as $\textbf{u}=\mbf/\varrho$ when $\varrho>0$) and it is enough to mollify the system in space variable. By using commutator estimate and Gr\"onwall's inequality we prove the pathwise uniqueness result. In deterministic setting the proof of energy conservation is done by mollifying the system. We consider only space mollification of $(\varrho,\mbf)$ and pass to the limit using suitable commutator estimate to show energy balance equation for \eqref{SCE1}--\eqref{SCE2}.  
 
At the end, we mention some of the important results in deterministic setting. For weak-strong uniqueness via relative entropy method, we refer to \cite{daf,dprnUnq} and its application to fluid mechanics to prove uniqueness and stability results \cite{chfrli,FeiKre,emil}. We also mention ill-posedness results of \cite{chdllkr,ls-WS1,frsle} for existence of infinitely many solutions for isentropic Euler system. Recently, uniqueness results with Besov regularity $B^{\al,\f}_q, \,\alpha>1/2$ and one-sided Lipschitz condition have been proved \cite{ghjn,ssajkk} for complete Euler system and general hyperbolic conservation laws. Onsager's conjecture \cite{onsager} on conservation of energy has been proved for incompressible \cite{cwt,fjwi}, compressible \cite{adsw,eyink,fgSGw} Euler system with solutions having regularity $B^{\al,\f}_3,\,\alpha>1/3$.  Entropy conservation is true for general system of conservation laws \cite{bgSGtw,bgSGtw-xtn}. 

 We organize rest of the article in the following way: in section \ref{sec:setting-defn} we give details on cylindrical Wiener process $W_t$ and define weak solutions for stochastic isentropic Euler system \eqref{SCE1}--\eqref{SCE2}. We put some preliminary results in section \ref{sec:prelim}. Our main result is stated in section \ref{sec:main}. In section \ref{sec:rel} a relative energy inequality is proved for \eqref{SCE1}--\eqref{SCE2}. In sections \ref{sec:weak-strong} and \ref{sec:energy-balance} we prove theorems on uniqueness and  energy balance respectively.

\section{Settings and Definitions}\label{sec:setting-defn}
We assume that a filtered  probability space $(\Omega,\mathfrak{F},\{\mathfrak{F}_t\}_{0\leq t\leq T},\mathbb{P})$, satisfying the following assumptions:
\begin{enumerate}[label=(H-\arabic*)]
	\item\label{H1} $\F_0$ contains all null sets $A\in\F$, s.t. $\Pp(A)=0$,
	\item\label{H2} $\F_t=\bigcap_{s>t}\F_s$.
\end{enumerate}

The noise, we consider in this article is similar to that of \cite{bfh-ren,bfh,bfh-book,bm}. We give a short description here, detailed discussions can be found in the mentioned references.

We consider the process is driven by a cylindrical Wiener process $\{W_t\}_{t\geq0}$,  on a filtered probability space $(\Omega,\F, \{\F_t\}_{t\geq0},\Pp)$ satisfying \ref{H1} and \ref{H2}, over a separable Hilbert space $\U$. Here $W_t$ is defined by the formal expansion
\begin{equation}\label{defn-W}
W_t:=\sum_{j=1}^{\infty}e_jW^j_t,
\end{equation}
where $\{e_j\}_{j\geq1}$ is an orthonormal basis of $\U$ and $\{W^j_t\}_{j\geq1}$ is a family of mutually independent real-valued Brownian motions with respect to $(\Omega,\F, \{\F_t\}_{t\geq0},\Pp)$. Let $\varrho\in\ L^2(\T^{N})$, $\varrho\geq0$ and $\mbf\in L^2(\T^d)$, then we define our diffusion coefficient $\G(\varrho,\mbf):\U\to L^2(\T^d;\R^d)$
\[\mathbb{G}(\varrho,\mbf)e_j:=\G_j(\cdot,\varrho(\cdot),\mbf(\cdot)),\]
where coefficients $\G_j=\G_j(x,\varrho,\mbf):\T^d\times [0,\infty)\times \R^d\to\R^d$ are $C^{1}$- functions and there exists a non-negative real numbers sequence $\{g_j\}$ such that the following holds uniformly in $x\in\T^d$,  
\begin{align}\label{con-G}
\begin{split}
 \G_j(\cdot,0,0)=0, |\partial_{\varrho}\G_j| + |\nabla_{\mbf}\G_j| \leq g_j\mbox{ and }\sum_{j\geq1} {g_j}^2<\infty.
\end{split}
\end{align}

Note that, when $\G$ satisfies \eqref{con-G} and $\varrho,\mbf$ are $\{\F_t\}$- progressively measurable $L^2(\T^d)$- valued process such that
\begin{align*}
 \varrho \in L^2\left(\Omega\times[0,T]; L^2(\T^d)\right)\mbox{ and } \mbf \in L^2\left(\Omega\times[0,T]; L^2(\T^d;\R^d)\right),
\end{align*}
then the following is a well-defined $\{\F_t\}$- martingale in $L^2(\T^d;\R^d)$, 
\[\int_0^t \G(\varrho,\mbf)\rd W_s := \sum_{j\geq1}\int_0^t \G_j(\cdot,\varrho,\mbf)\rd W^j_s.\]

Note that, the infinite sum in \eqref{defn-W} does not converge in   probabilistic sense as a random variable in $\U$. But, we can construct an auxiliary space $\U_0\supset\U$, where the sum converges. Define
\[\U_0:= \left\{u = \sum_{j\geq1}u_je_j;\ \sum_{j\geq1}\frac{u_j^2}{j^2}<\infty\right\},\]
and the norm
\[\|u\|_{\U_0}^2 := \sum_{j\geq1}\frac{u_j^2}{j^2}.\]
The embedding $\U\hookrightarrow\U_0$ is Hilbert-Schmidt and trajectories of $\{W_t\}$ are $\Pp$- a.s. in $C([0,T];\U_0)$.

Throughout the article, we work with the separable Hilbert space $\U=L^2(\dom)$.
\begin{definition}\label{def:pws}[pathwise weak solution]
	Let $(\Omega,\mathfrak{F},\{\mathfrak{F}_t\}_{0\leq t\leq T},\mathbb{P})$ be a filtered  probability space satisfying \ref{H1}--\ref{H2} and $\{W_t\}_{t\geq0}$ be an $\{\F_t\}_{0\leq t\leq T}$ - cylindrical Wiener process. Let $\G$ be satisfying condition \eqref{con-G}. We say, a triplet $[\varrho,\mbf,\tau]$ is a pathwise weak solution of \eqref{SCE1}--\eqref{SCE2} if the following holds:
	\begin{enumerate}
		\item The quantity $\tau>0$ is an a.s. strictly positive $\{\mathfrak{F}_t\}$-stopping time.
		\item The density $\varrho$ is $\{\F_t\}$- progressively measurable. There exists $\underline{r}>0$ such that the following holds for $\Pp$-a.s. 
		\begin{equation}\label{space:rho}
		\varrho\geq\underline{r},\ \ \text{and}\ \ \varrho\in C\left([0,\tau);L^{2}(\dom)\right)\cap L^{\f}([0,\tau)\times\dom).
		\end{equation}
		
		\item For each $\pmb{\varphi}\in C_c^{\f}(\dom,\R^d)$, the map  $t\mapsto\int\limits_{\dom}\mbf\cdot\pmb{\varphi}\,dx\in C([0,\tau))$ and the stochastic process $t\mapsto \int\limits_{\dom}\mbf\cdot\bvr\,dx$, is $(\F_t)$-  progressively measurable, such that, $\Pp$-a.s.
		\begin{equation}\label{space:m}
		\mbf\in C\left([0,\tau);L^{2}(\dom;\R^d)\right)\cap L^{\f}([0,\tau)\times\dom;\R^d).
		\end{equation}
		
		\item For all $\phi\in C_c^{\f}(\dom)$, the map $t\mapsto\int\limits_{\dom}\varrho\phi\, \rd x$, is $(\F_t)$-  progressively measurable, such that the following equation holds for $0\leq t_1\leq t_2<\tau$ $\Pp$-a.s.,
		\begin{equation}\label{eq:com:wk1}
		\int\limits_{\dom}\varrho(t_2,\cdot)\phi\,\rd x-	\int\limits_{\dom}\varrho(t_1,\cdot)\phi\,\rd x=\int\limits_{t_1}^{t_2}\int\limits_{\dom} \mbf(s,\cdot)\cdot\nabla_x\phi\,\rd x\rd s.
		\end{equation}
		
		\item For all $\bvr\in C_c^{\f}(\dom,\R^d)$ and the following equation holds for $0\leq t_1\leq t_2<\tau$ $\Pp$-a.s.,
		\begin{align}
		&\int\limits_{\dom}\mbf(t_2,\cdot)\cdot\bvr \,\rd x-\int\limits_{\dom}\mbf(t_1,\cdot)\cdot\bvr \,\rd x\nonumber\\
		&=\int\limits_{t_1}^{t_2}\int\limits_{\dom}\left[\left(\frac{\mbf\otimes\mbf}{\varrho}\right):\nabla_x\bvr+p(\varrho)\dv_x\bvr\right]\,\rd x\rd s+\int\limits_{t_1}^{t_2}\left(\int\limits_{\dom}\mathbb{G}(\varrho,\mbf)\cdot\bvr\,\rd x\right)\rd W_s.\label{eq:com:wk2}
		\end{align}
		
	\end{enumerate}
\end{definition}

\begin{remark}
	Note that integrability assumptions on $\varrho$ and $\mbf$ as in \eqref{space:rho}, \eqref{space:m} make integral equations \eqref{eq:com:wk1} and \eqref{eq:com:wk2} well-defined.
\end{remark}
\begin{remark}
	Note that strong solution defined as in \cite{bm} are included in the set of pathwise weak solution defined as in Definition \ref{def:pws}. We also remark that pathwise weak solution with sufficient regularity becomes strong solution as in \cite{bm}.
\end{remark}

\noi\textbf{Admissible criteria:} We say a weak solution to \eqref{SCE1}--\eqref{SCE2} is admissible if it satisfies the following energy inequality $\Pp$--a.s.
\begin{align}
\int\limits_{\dom}\left[\frac{\abs{\mbf(t_2)}^2}{2\varrho(t_2)}+P(\varrho(t_2))\right]\,\rd x&\leq  \int\limits_{\dom}\left[\frac{\abs{\mbf(t_1)}^2}{2\varrho(t_1)}+P(\varrho(t_1))\right]\,\rd x\nonumber\\
&+\int\limits_{t_1}^{t_2} \int\limits_{\dom}\frac{1}{2}\sum\limits_{k\geq1}\frac{\abs{\G(\varrho,\mbf)(e_k)}^2}{\varrho}\,\rd x\rd s+\int\limits_{t_1}^{t_2}\int\limits_{\dom}\G(\varrho,\mbf)\cdot\frac{\mbf}{\varrho}\rd x\rd W_s,\label{ineq:energy_integral}
\end{align}
for $0\leq t_1\leq t_2< \tau$, where $P(\varrho)$ is defined as follows
\begin{equation}\label{def:P}
P(\varrho):=\varrho\int\limits_{1}^{\varrho}\frac{p(z)}{z^2}\,\rd z.
\end{equation}

\noi\textbf{Remark on $P(\varrho)$:} Throughout this paper we work with solutions such that the density component is uniformly away from vacuum. Note that in the system \eqref{SCE2} the pressure term is assumed to have the following form $p(\varrho)=\kappa \varrho^{\gamma}$ for $\gamma>1$. For away from vacuum region, $P$ is thrice continuously differentiable. From definition of $P$, it can be checked that
\begin{equation}\label{derivative:P}
\varrho P^{\p}(\varrho)=P(\varrho)+p(\varrho)\mbox{ and }\varrho P^{\p\p}(\varrho)=p^{\p}(\varrho).
\end{equation}
\begin{remark}
	By taking expectation on both side of \eqref{ineq:energy_integral} we get
	\begin{align}
	\E\left(\int\limits_{\dom}\left[\frac{\abs{\mbf(t_2)}^2}{2\varrho(t_2)}+P(\varrho(t_2))\right]\rd x\right)&\leq  \E\left(\int\limits_{\dom}\left[\frac{\abs{\mbf(t_1)}^2}{2\varrho(t_1)}+P(\varrho(t_1))\right]\rd x\right)\nonumber\\
	&+\E\left(\int\limits_{t_1}^{t_2} \int\limits_{\dom}\frac{1}{2}\sum\limits_{k\geq1}\frac{\abs{\G(\varrho,\mbf)(e_k)}^2}{\varrho}\,\rd x\rd s\right),
	%
	\end{align}
	for $0\leq t_1\leq t_2< \tau$. Suppose the initial energy (at $t_1=0$) and the term generated from diffusion coefficient $\G(\varrho,\mbf)$ are finite, that is,
	\begin{equation}
	\E\left(\int\limits_{\dom}\left[\frac{\abs{\mbf_0}^2}{2\varrho_0}+P(\varrho_0)\right]\,\rd x\right)<\f\mbox{ and }	\E\left(\int\limits_{0}^{t} \int\limits_{\dom}\frac{1}{2}\sum\limits_{k\geq1}\frac{\abs{\G(\varrho,\mbf)(e_k)}^2}{\varrho}\,\rd x\rd s\right)<\f,
	\end{equation}
	for $0\leq t< \tau$. Then, 
	\begin{equation}
	\E\left(\int\limits_{\dom}\left[\frac{\abs{\mbf(t)}^2}{2\varrho(t)}+P(\varrho(t))\right]\rd x\right)<\f\mbox{	for }0\leq t< \tau.
	\end{equation}

\end{remark}

\section{Main results}\label{sec:main}
	Now we are ready to state main results of the article.
	\begin{theorem}[Pathwise uniqueness]\label{theorem_uniquess_I}
		Let $[\varrho,\mbf,\tau], [r,\nbf,\tau]$ be two pathwise weak solutions of \eqref{SCE1}--\eqref{SCE2} as in Definition \ref{def:pws}. Suppose $[\varrho,\mbf,\tau]$ additionally satisfies the energy inequality \eqref{ineq:energy_integral}. Let $[r,\nbf, \tau]$ be satisfying the following conditions:
		\begin{enumerate}
			\item $\alpha$-regularity: $(r,\nbf)\in L^2(\de,T;B^{\al,\f}_q(\dom,\R_+\times\R^d))$ for all $\de>0$ with $\al>1/2$ and $q\geq2$ holds $\mathbb{P}$--a.s.
			\item One-sided Lipschitz condition for velocity: there exists a function $\chi\in L^1(0,T)$ for $T>0$ such that the following holds for $\Pp$--a.s.
			\begin{equation}\label{ineq:one_sided}
			\int\limits_{\dom}-\frac{\nbf\cdot\xi}{r}\xi\cdot\nabla_x\varphi+\chi(t)\abs{\xi}^2\varphi\,\rd x\geq0,
			\end{equation} 
			for $\xi\in\R^d$, $0\leq t< \tau\wedge T$ and $0\leq\varphi\in C^\f_c(\dom)$.
			
		\end{enumerate}
	Then the following holds $\Pp$--a.s.,
	\begin{equation}
	\varrho=r,\,\mbf=\nbf\mbox{ for a.e. }(t,x)\in [0,\tau\wedge T]\times\dom.
	\end{equation}
	\end{theorem}
	
	\begin{theorem}[Energy balance]\label{theorem:energy_balance}
		Let $[\varrho,\mbf,\tau]$ be a pathwise weak solution of \eqref{SCE1}--\eqref{SCE2} according to Definition \ref{def:pws}. Assume that $(\varrho,\mbf)\in L^3([0,\tau);B^{\al,\f}_{3}(\dom,\R_+\times\R^d))$ for $\mathbb{P}$--a.s. and $\al>1/3$. Then, the following holds $\mathbb{P}$--a.s.,
		\begin{align}
		&\int\limits_{0}^{\tau}\pa_t\eta \int\limits_{\dom}\left(\frac{\abs{\mbf}^2}{2\varrho}+P(\varrho)\right)\varphi\,\rd x\rd s
		+\int\limits_{0}^{\tau}\eta \int\limits_{\dom}\left({\mbf}\left(\frac{\abs{\mbf}^2}{2\varrho^2}+P^{\p}(\varrho)\right)\right)\cdot\nabla_x\varphi\,\rd x\rd s\nonumber\\
		&+\int\limits_{0}^{\tau}\eta \int\limits_{\dom}\frac{1}{2}\sum\limits_{k\geq1}\frac{\abs{\G(\varrho,\mbf)(e_k)}^2}{\varrho}\varphi\,\rd x\rd s+\int\limits_{0}^{\tau}\eta \int\limits_{\dom}\G(\varrho,\mbf)\cdot\frac{\mbf}{\varrho}\varphi\rd x\rd W_s=0\label{integral_energy},
		\end{align}
		for all $\eta\in C_c^{\f}(\R_+)$ and $\varphi\in C_c^{\f}(\dom)$.

	\end{theorem}
%
%
\section{Preliminaries}\label{sec:prelim}
We first state a technical lemma which is used in proofs of main results. 

\begin{lemma}\label{lemma:limit_pass_diff}
		Let $q\geq2$. Let $Y_t$ be a stochastic process on $(\Omega,\mathfrak{F},(\mathfrak{F}_t)_{t\geq0},\Pp)$ such that
	\begin{equation}\label{integrability:Y}
	Y_t\in L^{\f}(0,T;L^{q}(\dom)),\mbox{ and }\E\left[\sup\limits_{t\in[0,T]}\norm{Y_t}^{2\la}_{L^{q}}\right]<\f\mbox{ for }1\leq\lambda<\f.
	\end{equation}
	Let $D_s$ be progressively measurable with $D_s\in L^{2}(\Omega;L^2(0,T;L_2(\U;L^{\frac{q}{q-1}}(\dom)))$ and 
	\begin{equation}\label{integrability:D}
\E\left(\sum\limits_{k\geq1}\int\limits_{0}^{T}\norm{D_s(e_k)}_{L^{\frac{q}{q-1}}(\dom)}^2\,\rd s\right)^{\la}<\f\mbox{ for }1\leq \la<\f.
	\end{equation}
	Let $\{\zeta_\e\}$ be the standard mollifiers sequence in space variable and we denote $h^\e=h*\zeta_\e$. Then we have the following up to a subsequence,
	\begin{equation}
	\int\limits_{0}^{T}\left(\int\limits_{\dom} Y_s^\e D_s^\e\rd x\right)\rd W_s\rr\int\limits_{0}^{T}\left(\int\limits_{\dom} Y_s D_s\rd x\right)\rd W_s\mbox{ as }\e\rr0\mbox{ in }\Pp-\mbox{a.s.}
	\end{equation}
\end{lemma}
We note that proof of Lemma \ref{lemma:limit_pass_diff} follows from a standard application of B\"urkhholder-Davis-Gundy inequality \cite{bfh-book,gm}. For the sake of completeness we give a proof in Appendix.

\begin{lemma}\label{lemma:limit_pass_drift1}
	Let $X_t,Y_t$ be a stochastic process on $(\Omega,\mathfrak{F},(\mathfrak{F}_t)_{t\geq0},\Pp)$ such that
	\begin{align}
	& X_t\in L^{\f}(0,T;L^{p}(\dom,\R^N)),\mbox{ and }\E\left[\sup\limits_{t\in[0,T]}\norm{X_t}^{2\la}_{L^{p}}\right]<\f\mbox{ for }1\leq\lambda<\f,\\
	& Y_t\in L^{\f}(0,T;L^{q}(\dom,\R^N)),\mbox{ and }\E\left[\sup\limits_{t\in[0,T]}\norm{Y_t}^{2\la}_{L^{q}}\right]<\f\mbox{ for }1\leq\lambda<\f,
	\end{align}
	such that $1\leq p,q\leq \f$ satisfying $p^{-1}+q^{-1}=1$. 
	Let $\{\zeta_\e\}_{\e>0}$ be a standard mollifiers sequence in space variable and we denote $h^\e=h*\zeta_\e$. Then we have the following up to a subsequence,
	\begin{equation}
	\int\limits_{0}^{T}\int\limits_{\dom}X_s^{\e}\cdot Y_s\,\rd x\rd s\rr	\int\limits_{0}^{T}\int\limits_{\dom}X_s\cdot Y_s\,\rd x\rd s
	\mbox{ as }\e\rr0\mbox{ in }\Pp-\mbox{a.s.}
	\end{equation}
\end{lemma}
We omit the proof of Lemma \ref{lemma:limit_pass_drift1} and it follows from standard properties of $L^p$ functions.

%
Next we state two commutator estimate lemmas: one is useful for proving energy balance equation and other one is helpful in the context of pathwise uniqueness.
      \begin{lemma}[Commutator estimate I, \cite{cwt,fgSGw,gmSG}]\label{lemma:commutator}
		Let $N,M\in\mathbb{N}$ and $\mathscr{O},\mathscr{O}_1\subset\R^d$ be two open bounded sets satisfying $\bar{\mathscr{O}}\subset\mathscr{O}_1$. Let $g\in B^{\al,\f}_3(\mathscr{O}_1,\R^M)$ and $h\in B^{\B,\f}_3(\mathscr{O}_1,\R^{M})$ for $\al,\B\in(0,1)$. Let $\mathcal{F}:\mathscr{U}\rr\R^{N\times M}$ be a $C^2$ function where $\mathscr{U}$ is the smallest open convex set containing closure of range of $g$. Then we have
		\begin{equation}
		\|\left(\mathcal{F}(g_\e)-\mathcal{F}(g)_\e\right):\nabla_x h_\e\|_{L^{1}(\mathscr{O})}\leq C_0\abs{g}_{B^{\al,\f}_3(\mathscr{O}_1)}^2\abs{h}_{B^{\B,\f}_3(\mathscr{O})}\e^{2\al+\B-1},
		\end{equation}
		where $C_0$ depends on $\sup\{\abs{\nabla_u^2\mathcal{F}};\,u\in\mathscr{U}\}$ and domain $\mathscr{O}$.
\end{lemma}
We omit the proof of Lemma \ref{lemma:commutator}. It can be found in \cite{fgSGw,gmSG}.
      \begin{lemma}[Commutator estimate II,  \cite{fgj}]\label{lemma:commutator1}
	Let $N,M\in\mathbb{N}$ and $\mathscr{V},\mathscr{V}_1\subset\R^d$ be open sets such that $\bar{\mathscr{V}}\subset\mathscr{V}_1$. Let $\vartheta\in B^{\al,\f}_q(\mathscr{V}_1,\R^M)$ for some $\al\in(0,1)$ and $q\geq2$. Let $\mathscr{W}$ be an open convex set such that $\overline{Image(\vartheta)}\subset\mathscr{W}$. Let $\mathcal{Q}:\mathscr{W}\rr\R^{N\times M}$ be a $C^2$ function. Then we have
	\begin{equation}
	\|\nabla_x\left(\mathcal{Q}(\vartheta_\e)-\mathcal{Q}(\vartheta)_\e\right)\|_{L^{q/2}(\mathscr{V})}\leq C_0\abs{\vartheta}_{B^{\al,\f}_q(\mathscr{V})}^2\e^{2\al-1},
	\end{equation}
	where $C_0$ depends on $\sup\{\abs{\nabla_u^2\mathcal{Q}};\,u\in\mathscr{W}\}$ and domain $\mathscr{V}$.
\end{lemma}
The proof of Lemma \ref{lemma:commutator1} is omitted here and it can be found in \cite{fgj}.
	
\section{Relative energy inequality}\label{sec:rel}	
	In this section we prove a technical lemma  for proving uniqueness result and it also sets a ground to prove energy balance as well. In order to prove both the theorems 
	we mollify solutions only in space variable we use the following form of relative energy for the system \eqref{SCE1}--\eqref{SCE2}, 
	\begin{equation}
	\mathcal{E}(\varrho,\mathbf{m}|r,\nbf):=\int\limits_{\dom}\left(\frac{\varrho}{2}\abs{\frac{\mbf}{\varrho}-\frac{\nbf}{r}}^2+P(\varrho)-(\varrho-r)P^{\p}(r)-P(r)\right)\rd x.
	\end{equation}
	Note that it measures the distance between two weak solutions. If $(\varrho,\mbf)$ is a weak solution satisfying admissible criteria and $(r,\nbf)$ is a process having regularity in space then we can have a estimate for $\mathcal{E}(\varrho,\mathbf{m}|r,\nbf)(t_2)-\mathcal{E}(\varrho,\mathbf{m}|r,\nbf)(t_1)$ for two different time $0\leq t_1\leq t_2<\tau$. More precisely, we prove the following result:
	
	\begin{proposition}[Relative energy inequality]\label{prop:rel_energy} 	Let $[\varrho,\mbf,\tau]$ be a pathwise weak solution satisfying energy inequality \eqref{ineq:energy_integral}. Let $r,\nbf $ be two stochastic processes on $(\Omega,\F,\{\F_t\}_{t\geq0},\Pp)$ determined by
	\begin{align}
	\rd r&=F_s\rd s+A_s\rd W_s,\label{SDE:r}\\
	\rd \nbf&=G_s\rd s+B_s\rd W_s,\label{SDE:n}
	\end{align}
	where $F_s,G_s,A_s,B_s$ are progressively measurable such that the following holds for all $1\leq q<\f$,
	\begin{equation}
	\begin{array}{lll}
	&F_s\in L^{q}(\Omega;L^q(0,T;C^{1}(\dom))),\,&A_s\in L^{2}(\Omega;L^2(0,T;L_2(\U;L^{2}(\dom))),\\
	&G_s\in L^{q}(\Omega;L^q(0,T;C^1(\dom,\R^d))),\,&B_s\in L^{2}(\Omega;L^2(0,T;L_2(\U;L^{2}(\dom,\R^d)))),\\
%
	&\E\left[\int\limits_{0}^{T}\sum\limits_{k\geq1}\abs{A_s(e_k)}^q\rd x\rd s\right]<\f,	&\E\left[\int\limits_{0}^{T}\sum\limits_{k\geq1}\abs{B_s(e_k)}^q\rd x\rd s\right]<\f.
		\end{array}	
	\end{equation}
    We assume that
    \begin{align}
    r\in C([0,T];C^1(\dom)),\,\nbf\in C([0,T];C^1(\dom,\R^d)),\\
    \E\left(\sup\limits_{0\leq t\leq T}\|r\|_{C^1(\dom)}^2\right)^q+\E\left(\sup\limits_{0\leq t\leq T}\|\nbf\|_{C^1(\dom,\R^d)}^2\right)^q\leq c(q)
    \end{align}
    for some constant $c(q)$ depending on $q$  for all $1\leq q<\f$. Then for $0\leq t_1\leq t_2< \tau$, we have
	\begin{equation}
	\mathcal{E}(\varrho,\mbf|r,\nbf)(t_2)\leq \mathcal{E}(\varrho,\mbf|r,\nbf)(t_1)+\int\limits_{t_1}^{t_2}M_{rem}\rd W_s+\int\limits_{t_1}^{t_2}\mathcal{R}(\varrho,\mbf|r,\nbf)(s)\,\rd s,
	\end{equation}
	where $\mathcal{R}(\varrho,\mbf|r,\nbf)(s)$ is defined as 
	\begin{align}
	\mathcal{R}(\varrho,\mbf|r,\nbf)(s)&:=\int\limits_{\dom}\left((r-\varrho) P^{\p\p}(r)+\frac{\mbf\cdot\nbf}{r^2}-\frac{\varrho\abs{\nbf}^2}{r^3}\right)F_s
	+\frac{\varrho}{r}\left(\frac{\nbf}{r}-\frac{\mbf}{\varrho}\right)\cdot G_s\,\rd x\nonumber\\
	&+\int\limits_{\dom}\frac{1}{2}(P^{\p\p}(r)+P^{\p\p\p}(r)(r-\varrho))\sum\limits_{k=1}^{\f}\abs{A_s(e_k)}^2\,\rd x\nonumber\\
	&+\int\limits_{\dom}\left(\sum\limits_{k=1}^{\f}\abs{A_s(e_k)}^2\left(\frac{3\varrho\abs{\nbf}^2}{2r^4}-\frac{\mbf\cdot\nbf}{r^3}\right)	-\sum\limits_{k\geq1}A_s(e_k)B_s(e_k)\cdot\left(\frac{2\varrho\nbf}{r^3}-\frac{\mbf}{r^2}\right)\right)\rd x\nonumber\\
	&+\int\limits_{\dom}\sum\limits_{k\geq1}\left(\frac{\nbf\cdot \G(\varrho,\mbf)(e_k) }{r^2}A_s(e_k)+\frac{\varrho}{2}\abs{\frac{B_s(e_k)}{r}-\frac{\G(\varrho,\mbf)(e_k)}{\varrho}}^2\right)\rd x\nonumber\\
	&+\int\limits_{\dom}\left(\mbf\otimes\left(\frac{\nbf}{r}-\frac{\mbf}{\varrho}\right):\nabla_x\frac{\nbf}{r}-p(\varrho)\dv_x\frac{\nbf}{r}-\mbf\cdot\nabla_xP^{\p}(r)\right)\rd x,
	\end{align}
	and the process $M_{rem}$ is defined as 
	\begin{align}
	M_{rem}&:=\int\limits_{\dom}\left(\left(\frac{\nbf}{r^2}A_s-\frac{B_s}{r}\right)\cdot\mbf+\G(\varrho,\mbf)\cdot\left(\frac{\mbf}{\varrho}-\frac{\nbf}{r}\right)-\varrho\frac{\abs{\nbf}^2}{r^3}A_s+\varrho\frac{B_s\cdot\nbf}{r^2}\right)\rd x\nonumber\\
	&+\int\limits_{\dom}(r-\varrho) P^{\p\p}(r)A_s\,\rd x.
	\end{align}
	\end{proposition}
	Though we consider the general case for $r$ with $A_s$ (possibly $\neq0$), in the proof of Theorem \ref{theorem_uniquess_I} we use Proposition \ref{prop:rel_energy} for a weak solution $(r,\nbf)$, in that situation $A_s=0$. An analogous version of Proposition \ref{prop:rel_energy} is true for stochastic Navier-Stokes equation (see \cite{bfh-ren}).

\begin{proof}[Proof of Proposition \ref{prop:rel_energy}]
		From definition of relative entropy we get
	\begin{align}
	\mathcal{E}(\varrho,\mbf|r,\nbf)=\int\limits_{\dom}\left(\frac{\abs{\mbf}^2}{2\varrho}-\frac{\mbf\cdot\nbf}{ r}+
	\frac{\varrho\abs{\nbf}^2}{2r^2}+P(\varrho)-\varrho P^{\p}(r)+P^{\p}(r)r-P(r)\right)\rd x.
	\end{align}
	We divide the proof into the following steps.
	\noi\begin{enumerate}[label=Step-\arabic*]
%
		\item Since $\varrho,\mbf$ do not have enough regularity to make a pointwise sense of the system \eqref{SCE1}--\eqref{SCE2}, we mollify the system for $(\varrho,\mbf)$ in the space variable and get the following
		\begin{align}
		\rd \varrho^\e&=-\dv_x\mbf^\e\rd s,\\
		\rd \mbf^\e&=\left[-\dv_x\left(\frac{\mbf\otimes\mbf}{\varrho}\right)^\e-\nabla_xp(\varrho)^\e\right]\rd s+\G(\varrho,\mbf)^\e\rd W_s.
		\end{align}
	
		\item Note that $r\geq \underline{r}>0$. We apply It\^{o}'s formula \cite{ito} on the process $r$ for the function $r\mapsto1/r$ to obtain
		\begin{align}
		\rd\left(\frac{1}{r}\right)&=\left[-\frac{1}{r^2} F_s+\frac{1}{r^3}\sum\limits_{k\geq1}\abs{A_s(e_k)}^2\right]\rd s-\frac{1}{r^2}A_s\rd W_s.
		\end{align}By using product rule for processes $\nbf$ and $1/r$ we get
		\begin{align}
		\rd\left(\frac{\nbf}{r}\right)&=\left[-\frac{1}{r^2} F_s+\frac{1}{r^3}\sum\limits_{k\geq1}\abs{A_s(e_k)}^2\right]\nbf\rd s+\frac{G_s}{r}\rd s\nonumber\\
		&-\frac{1}{r^2}\sum\limits_{k\geq1}A_s(e_k)B_s(e_k)\rd s+\left[-\frac{\nbf}{r^2}A_s+\frac{B_s}{r}\right]\rd W_s.
		\end{align}
		\item Again by using product rule for processes $\mbf^\e$ and $\nbf/r$ we have
		\begin{align}
		\begin{split}
		\rd\left(\frac{\mbf^\e\cdot\nbf}{r}\right)&=\left[-\frac{1}{r^2} F_s+\frac{1}{r^3}\sum\limits_{k\geq1}\abs{A_s(e_k)}^2\right]\mbf^\e\cdot\nbf\rd s+\frac{G_s\cdot\mbf^\e}{r}\rd s\\
		&-\frac{1}{r^2}\sum\limits_{k\geq1}A_s(e_k)B_s(e_k)\cdot\mbf^\e\rd s+\left[-\dv_x\left(\frac{\mbf\otimes\mbf}{\varrho}\right)^\e-\nabla_xp(\varrho)^\e\right]\cdot\frac{\nbf}{r}\rd s\\
		&+\sum\limits_{k\geq1}\left[\left[-\frac{\nbf}{r^2}A_s(e_k)+\frac{B_s(e_k)}{r}\right]\cdot\G(\varrho,\mbf)^\e(e_k)\right]\rd s\\
				&+\left[\left[-\frac{\nbf}{r^2}A_s+\frac{B_s}{r}\right]\cdot\mbf^\e+\G(\varrho,\mbf)^\e\cdot\frac{\nbf}{r}\right]\rd W_s.
		\end{split}\label{rel1}
		\end{align}
		
		\item Then, applying It\^{o}'s formula for the function $u\mapsto \frac{\abs{u}^2}{2}$ with respect to $\frac{\nbf}{r}$, we get
		\begin{align*}
		\rd\left(\frac{\abs{\nbf}^2}{2r^2}\right)&=\left[-\frac{1}{r^2} F_s+\frac{1}{r^3}\sum\limits_{k\geq1}\abs{A_s(e_k)}^2\right]\frac{\abs{\nbf}^2}{r}\rd s+\frac{G_s\cdot\nbf}{r^2}\rd s-\frac{1}{r^2}\sum\limits_{k\geq1}A_s(e_k)\frac{B_s(e_k)\cdot\nbf}{r}\rd s\\
		&+\frac{1}{2}\sum\limits_{k\geq1}\left[\frac{\abs{\nbf}^2}{r^4}\abs{A_s(e_k)}^2-2A_s(e_k)\frac{B_s(e_k)\cdot\nbf}{r^3}+\frac{\abs{B_s(e_k)}^2}{r^2}\right]\rd s\\
		&+\left[-\frac{\abs{\nbf}^2}{r^3}A_s+\frac{B_s\cdot\nbf}{r^2}\right]\rd W_s.
		\end{align*}
		Now we use the product rule for $\varrho^\e$ and $\frac{\abs{\nbf}^2}{2r}$ to obtain
		\begin{align}
		\begin{split}
		\rd\left(\varrho^\e\frac{\abs{\nbf}^2}{2r^2}\right)&=\left[-\frac{1}{r^2} F_s+\frac{1}{r^3}\sum\limits_{k\geq1}\abs{A_s(e_k)}^2\right]\frac{\varrho^\e\abs{\nbf}^2}{r}\rd s
		-\frac{\varrho^\e}{r^2}\sum\limits_{k\geq1}A_s(e_k)\frac{B_s(e_k)\cdot\nbf}{r}\rd s\\
		&+\frac{\varrho^\e G_s\cdot\nbf}{r^2}\rd s+\frac{1}{2}\sum\limits_{k\geq1}\left[\frac{\abs{\nbf}^2}{r^4}\abs{A_s(e_k)}^2-2A_s(e_k)\frac{B_s(e_k)\cdot\nbf}{r^3}
		+\frac{\abs{B_s(e_k)}^2}{r^2}\right]\varrho^\e\rd s\\
		&-\dv_x\mbf^\e\frac{\abs{\nbf}^2}{2r^2}\rd s+\left[-\varrho^\e\frac{\abs{\nbf}^2}{r^3}A_s+\varrho^\e\frac{B_s\cdot\nbf}{r^2}\right]\rd W_s.
		\end{split}\label{rel2}
			\end{align}
		\item Again we apply It\^{o}'s formula for functions $r\mapsto P^{\p}(r)$ and $r\mapsto rP^{\p}(r)-P(r)$ respectively to obtain,
		\begin{align}
		\rd P^{\p}(r)&=P^{\p\p}(r)F_s\rd s+\frac{1}{2}P^{\p\p\p}(r)\sum\limits_{k=1}^{\f}\abs{A_s(e_k)}^2\rd s+P^{\p\p}(r)A_s\rd W_s,\\
		\rd (rP^{\p}(r)-P(r))&=rP^{\p\p}(r)F_s\rd s+\frac{1}{2}(P^{\p\p}(r)+rP^{\p\p\p}(r))\sum\limits_{k=1}^{\f}\abs{A_s(e_k)}^2\rd s+rP^{\p\p}(r)A_s\rd W_s.\label{rel3}
		\end{align}
		Therefore, by using the product rule for $\varrho^\e$ and $P^{\p}(r)$ we get
		\begin{align}
		\rd (\varrho^\e P^{\p}(r))&=\varrho^\e P^{\p\p}(r)F_s\rd s+\frac{1}{2}P^{\p\p\p}(r)\varrho^\e\sum\limits_{k=1}^{\f}\abs{A_s(e_k)}^2\rd s\nonumber\\
		&-\dv_x\mbf^\e P^{\p}(r)\rd s+\varrho^\e P^{\p\p}(r)A_s\rd W_s.\label{rel4}
		\end{align}
		
		\item Clubbing \eqref{rel1}, \eqref{rel2}, \eqref{rel3} and \eqref{rel4} we obtain 
		\begin{align}
			\begin{split}
				&\rd \left(-\frac{\mbf^\e\cdot\nbf}{r}+\frac{\varrho^\e\abs{\nbf}^2}{2r^2}-P^\p(r)\varrho^\e+P^\p(r)r-P(r)\right)\\
				&=\left(\mathcal{T}_1^\e+\mathcal{T}_2^\e+\mathcal{T}_3^\e+\mathcal{T}_{4}^\e\right)\rd s+\mathcal{S}_{rem}^{\e}\rd W_s,
			\end{split}\label{rel-5}
		\end{align}
		where $\mathcal{T}_j^\e,\,j=1,2,3,4$ and $\mathcal{S}_{rem}^\e$ are defined as
		\begin{align}
		\mathcal{T}_1^\e&:=\left((r-\varrho^\e) P^{\p\p}(r)+\frac{\mbf^\e\cdot\nbf}{r^2}-\frac{\varrho^\e\abs{\nbf}^2}{r^3}\right)F_s+\frac{G_s\cdot(\varrho^{\e}\nbf-r\mbf^\e)}{r^2},
		\label{def:T1}\\
	\mathcal{T}_2^\e&:=\frac{1}{2}(P^{\p\p}(r)+P^{\p\p\p}(r)(r-\varrho^\e))\sum\limits_{k=1}^{\f}\abs{A_s(e_k)}^2\nonumber\\
	&+\sum\limits_{k=1}^{\f}\abs{A_s(e_k)}^2\left(\frac{3\varrho^\e\abs{\nbf}^2}{2r^4}-\frac{\mbf^\e\cdot\nbf}{r^3}\right)	-\sum\limits_{k\geq1}A_s(e_k)B_s(e_k)\cdot\left(\frac{2\varrho^\e\nbf}{r^3}-\frac{\mbf^\e}{r^2}\right),\label{def:T2}\\
	\mathcal{T}_3^\e&:=\sum\limits_{k\geq1}\left(\left(\frac{\nbf}{r^2}A_s(e_k)-\frac{B_s(e_k)}{r}\right)\cdot\G(\varrho,\mbf)^\e(e_k)+\frac{1}{2}\frac{\varrho^\e\abs{B_s(e_k)}^2}{r^2}\right),\label{def:T3}\\
		\mathcal{T}_{4}^\e&:=-\dv_x\mbf^\e\frac{\abs{\nbf}^2}{2r^2}+\left(\dv_x\left(\frac{\mbf\otimes\mbf}{\varrho}\right)^\e+\nabla_xp(\varrho)^\e\right)\cdot\frac{\nbf}{r}+\dv_x\mbf^\e P^{\p}(r),	\label{def:T-det}\\
	\mathcal{S}_{rem}^\e&:=\left(\frac{\nbf}{r^2}A_s-\frac{B_s}{r}\right)\cdot\mbf^\e-\G(\varrho,\mbf)^\e\cdot\frac{\nbf}{r}-\varrho^\e\frac{\abs{\nbf}^2}{r^3}A_s+\varrho^\e\frac{B_s\cdot\nbf}{r^2}+(r-\varrho^\e )P^{\p\p}(r)A_s.\label{def:M_1}
		\end{align}
		
		\item Now we integrate both side of \eqref{rel-5} over $\dom$ and by using stochastic Fubini's theorem (see \cite[chapter 4]{dpZ}) to obtain,
		\begin{align}
		&\rd\left(\int\limits_{\dom}\left(-\frac{\mbf^\e\cdot\nbf}{r}+\frac{\varrho^\e\abs{\nbf}^2}{2r^2}-P^\p(r)\varrho^\e+P^\p(r)r-P(r)\right) \rd x\right)\nonumber\\
		&=\left(\mathcal{R}_1^{\e}+\mathcal{R}_2^{\e}+\mathcal{R}_3^{\e}+\mathcal{R}_{4}^\e\right)\rd s+M_{rem}^\e\rd W_s,\label{rel-2}
		\end{align}
		where $\mathcal{R}_j^\e$ for $j=1,2,3,4$, and ${M}_{rem}^\e$ are defined as follows,
		\begin{align}
		\mathcal{R}_j^{\e}:=\int\limits_{\dom}\mathcal{T}_j^\e\,dx\mbox{ for }j=1,2,3,4,
		\mbox{ and }M_{rem}^\e:=\int\limits_{\dom}\mathcal{S}^\e_{rem}\,dx.
		\end{align}
		Using integration by parts we have
		\begin{align}
			\mathcal{R}_{4}^\e&=\int\limits_{\dom}\left(-\dv_x\mbf^\e\frac{\abs{\nbf}^2}{2r^2}+\left[\dv_x\left(\frac{\mbf\otimes\mbf}{\varrho}\right)^\e+\nabla_xp(\varrho)^\e\right]\cdot\frac{\nbf}{r}+\dv_x\mbf^\e P^{\p}(r)\right)\rd x\nonumber\\
			&=\int\limits_{\dom}\left(\left(\mbf^\e\otimes\frac{\nbf}{r}-\left(\frac{\mbf\otimes\mbf}{\varrho}\right)^\e\right):\nabla_x\frac{\nbf}{r}-p(\varrho)^\e\dv_x\frac{\nbf}{r}-\mbf^\e\cdot\nabla_xP^{\p}(r)\right)\rd x.\label{term:R-det}
		\end{align}
		
		\item By using Lemma \ref{lemma:limit_pass_drift1} we obtain the following limits  a.s. in $\Pp$ as $\e\rr0$,
			\begin{align}
		\mathcal{R}_1^\e&\rr\int\limits_{\dom}\left((r-\varrho) P^{\p\p}(r)+\frac{\mbf\cdot\nbf}{r^2}-\frac{\varrho\abs{\nbf}^2}{r^3}\right)F_s+\frac{\varrho G_s}{r}\cdot\left(\frac{\nbf}{r}-\frac{\mbf}{\varrho}\right),\label{limit:R1}\\
		\mathcal{R}^\e_2&\rr\int\limits_{\dom}\frac{1}{2}(P^{\p\p}(r)+P^{\p\p\p}(r)(r-\varrho))\sum\limits_{k=1}^{\f}\abs{A_s(e_k)}^2\rd x\nonumber\\
		&+\int\limits_{\dom}\left(\sum\limits_{k=1}^{\f}\abs{A_s(e_k)}^2\left(\frac{3\varrho\abs{\nbf}^2}{2r^4}-\frac{\mbf\cdot\nbf}{r^3}\right)	-\sum\limits_{k\geq1}A_s(e_k)B_s(e_k)\cdot\left(\frac{2\varrho\nbf}{r^3}-\frac{\mbf}{r^2}\right)\right)\rd x,\label{limit:R2}\\
		\mathcal{R}_3^\e&\rr\int\limits_{\dom}\sum\limits_{k\geq1}\left[\left[\frac{\nbf}{r^2}A_s(e_k)-\frac{B_s(e_k)}{r}\right]\cdot\G(\varrho,\mbf)(e_k)+\frac{1}{2}\frac{\varrho\abs{B_s(e_k)}^2}{r^2}\right]\rd x,
		\label{limit:R3}\\
		\mathcal{R}_{4}^\e&\rr\int\limits_{\dom}\left(\mbf\otimes\left(\frac{\nbf}{r}-\frac{\mbf}{\varrho}\right):\nabla_x\frac{\nbf}{r}-p(\varrho)\dv_x\frac{\nbf}{r}-\mbf\cdot\nabla_xP^{\p}(r)\right)\rd x.\label{limit:R-det}
		\end{align}
		Next we invoke Lemma \ref{lemma:limit_pass_diff} to get the following $\Pp$-a.s. as $\e\rr0$,
		\begin{equation}\label{limit:M}
				M_{rem}^\e\rr
				\int\limits_{\dom} \left(\frac{\varrho\nbf}{r^2}\cdot\left(\frac{\mbf}{\varrho}-\frac{\nbf}{r}\right)A_s+(r-\varrho) P^{\p\p}(r)A_s-\frac{\varrho}{r}\left(\frac{\mbf}{\varrho}-\frac{\nbf}{r}\right)\cdot B_s-\frac{\nbf}{r}\cdot\G(\varrho,\mbf)\right)\rd x.
		\end{equation}
		 By using Lemma \ref{lemma:limit_pass_drift1} we also get
		 \begin{align}
		 \begin{split}
		 				&\int\limits_{\dom}\left(-\frac{\mbf^\e\cdot\nbf}{r}+\frac{\varrho^\e\abs{\nbf}^2}{2r^2}-P^\p(r)\varrho^\e+P^\p(r)r-P(r)\right)\rd x\\
		 				&\rr\int\limits_{\dom}\left(-\frac{\mbf\cdot\nbf}{r}+\frac{\varrho\abs{\nbf}^2}{2r^2}-P^\p(r)\varrho+P^\p(r)r-P(r)\right)\rd x\mbox{ as }\e\rr0,\Pp-\mbox{a.s.}\label{limit:rel-partial}
		 \end{split}
		 \end{align}
	\end{enumerate}
Now clubbing \eqref{limit:R1}--\eqref{limit:rel-partial} and \eqref{rel-2} with \eqref{ineq:energy_integral} we conclude Proposition \ref{prop:rel_energy}.
\end{proof}

\section{Proof of pathwise uniqueness}\label{sec:weak-strong}
Now we are ready to prove the pathwise uniqueness for weak solution to the stochastic isentropic Euler system. We use relative energy inequality as derived in Proposition \ref{prop:rel_energy} and commutator estimate lemma. 
\begin{proof}[Proof of Theorem \ref{theorem_uniquess_I}:] We prove the theorem by using relative entropy inequality derived as in Proposition \ref{prop:rel_energy}. For that we need processes $r$ and $\nbf $ to be at least $C^1$ in space variable. Since $(r,\nbf)$ satisfies \eqref{SCE1}--\eqref{SCE2} in weak sense, we can not use Proposition \ref{prop:rel_energy} directly for $(r,\nbf)$, instead, we use a space-mollified version of $r,\nbf$. Now mollifying the system \eqref{SCE1}--\eqref{SCE2} in space variable for $(r,\nbf)$ we get
	\begin{align}
	\rd r^\e&=-\dv_x\nbf^\e\rd s,\\
	\rd \nbf^\e&=\left(-\dv_x\nbf^\e\frac{\nbf^\e}{r^\e}-\nbf^\e\cdot\nabla_x\frac{\nbf^\e}{r^\e}-p^{\p}(r^\e)\nabla_xr^\e+\mathcal{K}_1^\e+\mathcal{K}_2^\e\right)\rd s+\mathbb{G}(r,\nbf)^\e\rd W_s,
	\end{align}
	where $\mathcal{K}_1^\e$ and $\mathcal{K}_2^\e$ are defined as
	\begin{align}
	\mathcal{K}_1^\e&:=\dv_x\left(\frac{\nbf^\e\otimes\nbf^\e}{r^\e}\right)-\dv_x\left(\frac{\nbf\otimes\nbf}{r}\right)^\e,\\
	\mathcal{K}_2^\e&:=\nabla_xp(r^\e)-\nabla_xp(r)^\e.
	\end{align}
	We apply Proposition \ref{prop:rel_energy} with $r=r^\e,\nbf=\nbf^\e$ and
\begin{align*}
	&F_s=-\dv_x\nbf^\e,\\
	& A_s=0,\\
	 &G_s=-\dv_x\nbf^\e\frac{\nbf^\e}{r^\e}-\nbf^\e\cdot\nabla_x\frac{\nbf^\e}{r^\e}-p^{\p}(r^\e)\nabla_xr^\e+\mathcal{K}_1^\e+\mathcal{K}_2^\e,\\
	 & B_s=\mathbb{G}(r,\nbf)^\e.
\end{align*}
	Then we get	
	\begin{equation}\label{cal-rel1}
	\mathcal{E}(\varrho,\mbf|r^\e,\nbf^\e)(t_2)\leq \mathcal{E}(\varrho,\mbf|r^\e,\nbf^\e)(t_1)+\int\limits_{t_1}^{t_2} M_{rem}^\e\rd W_s+\int\limits_{t_1}^{t_2}\bar{\mathcal{R}}(\varrho,\mbf|r^\e,\nbf^\e)(s)\,\rd s,
	\end{equation}
	where $\bar{\mathcal{R}}(\varrho,\mbf|r^\e,\nbf^\e), M_{rem}^\e$ are defined as 
		\begin{align}
	\bar{\mathcal{R}}(\varrho,\mbf|r^\e,\nbf^\e)&:=-\int\limits_{\dom}\left((r^\e-\varrho) P^{\p\p}(r^\e)+\frac{\mbf\cdot\nbf^\e}{(r^\e)^2}-\frac{\varrho\abs{\nbf^\e}^2}{(r^\e)^3}+\frac{\nbf^\e\cdot(\varrho\nbf^\e-r^\e\mbf)}{(r^\e)^3}\right)\dv_x\nbf^\e\,\rd x\nonumber\\
	&-\int\limits_{\dom}\left(\varrho\frac{\nbf^\e}{r^\e}\cdot\nabla_x\frac{\nbf^\e}{r^\e}\left(\frac{\nbf^\e}{r^\e}-\frac{\mbf}{\varrho}\right)+p^{\p}(r^\e)\frac{\varrho}{r^\e}\nabla_xr^\e\cdot\left(\frac{\nbf^\e}{r^\e}-\frac{\mbf}{\varrho}\right)\right)\rd x\nonumber\\
	&+\int\limits_{\dom}\left(\frac{\varrho}{r^\e}\left(\frac{\nbf^\e}{r^\e}-\frac{\mbf}{\varrho}\right)\cdot\left(\mathcal{K}_1^\e+\mathcal{K}_2^\e\right)+
	\sum\limits_{k\geq1}\frac{\varrho}{2}\abs{\frac{\mathbb{G}(r,\nbf)^\e(e_k)}{r^\e}-\frac{\mathbb{G}(\varrho,\mbf)(e_k)}{\varrho}}^2\right)\rd x\nonumber\\
	&+\int\limits_{\dom}\left(\mbf\otimes\left(\frac{\nbf^\e}{r^\e}-\frac{\mbf}{\varrho}\right):\nabla_x\frac{\nbf^\e}{r^\e}-p(\varrho)\dv_x\frac{\nbf^\e}{r^\e}-\mbf\cdot\nabla_xP^{\p}(r^\e)\right)\rd x
	\end{align}
	and 
    \begin{align}
	M_{rem}^\e&:=-\int\limits_{\dom}\varrho\left(\frac{\G(r,\nbf)^\e}{r^\e}-\frac{\G(\varrho,\mbf)}{\varrho}\right)\cdot\left(\frac{\mbf}{\varrho}-\frac{\nbf^\e}{r^\e}\right)\,\rd x.
	\end{align}
		By using $P^{\p\p}(r^\e)=p^{\p}(r^\e)/r^\e$, we further simplify 
		\begin{align}
	\bar{\mathcal{R}}(\varrho,\mbf|r^\e,\nbf^\e)&=\int\limits_{\dom}\left(-p^{\p}(r^\e)\dv_x\nbf^\e
+\sum\limits_{k\geq1}\frac{\varrho}{2}\abs{\frac{\mathbb{G}(r,\nbf)^\e(e_k)}{r^\e}-\frac{\mathbb{G}(\varrho,\mbf)(e_k)}{\varrho}}^2\right)\,\rd x\nonumber\\
	&-\int\limits_{\dom}\left(\varrho\left(\frac{\nbf^\e}{r^\e}-\frac{\mbf}{\varrho}\right)\otimes\left(\frac{\nbf^\e}{r^\e}-\frac{\mbf}{\varrho}\right):\nabla_x\frac{\nbf^\e}{r^\e}+\left(p(\varrho)-\varrho p^\p(r^\e)\right)\dv_x\frac{\nbf^\e}{r^\e}\right)\rd x\nonumber\\
		&+\int\limits_{\dom}\frac{\varrho}{r^\e}\left(\frac{\nbf^\e}{r^\e}-\frac{\mbf}{\varrho}\right)\cdot\left(\mathcal{K}_1^\e+\mathcal{K}_2^\e\right)\rd x.\label{cal:R-1}
	\end{align}
	Applying integration by parts twice we obtain the following
	\begin{align}
	\int\limits_{\dom}\left(r^\e p^\p(r^\e)-p(r^\e)\right)\dv_x\frac{\nbf^\e}{r^\e}\,\rd x&=-\int\limits_{\dom}\frac{\nbf^\e}{r^\e}\cdot\nabla_x\left(r^\e p^\p(r^\e)-p(r^\e)\right)\,\rd x=-\int\limits_{\dom}p^{\p\p}(r^\e)\nbf^\e\cdot\nabla_xr^\e\,\rd x\nonumber\\
	&=-\int\limits_{\dom}\nbf^\e\cdot\nabla_xp^{\p}(r^\e)\,\rd x=\int\limits_{\dom}p^{\p}(r^\e)\dv_x\nbf^\e\,\rd x.\label{cal:P-1}
	\end{align}
   By using \eqref{cal:P-1} in \eqref{cal:R-1} we obtain following
			\begin{align}\label{cal-R}
	\bar{\mathcal{R}}(\varrho,\mbf|r^\e,\nbf^\e)(t)&=-\int\limits_{\dom}\varrho\left(\frac{\nbf^\e}{r^\e}-\frac{\mbf}{\varrho}\right)\otimes\left(\frac{\nbf^\e}{r^\e}-\frac{\mbf}{\varrho}\right):\nabla_x\frac{\nbf^\e}{r^\e}\,\rd x\nonumber\\
	&-\int\limits_{\dom}\left(p(\varrho)-(\varrho-r^\e) p^\p(r^\e)-p(r^\e)\right)\dv_x\frac{\nbf^\e}{r^\e}\,\rd x\nonumber\\
	&	+\int\limits_{\dom}\sum\limits_{k\geq1}\frac{\varrho}{2}\abs{\frac{\mathbb{G}(r,\nbf)^\e(e_k)}{r^\e}-\frac{\mathbb{G}(\varrho,\mbf)(e_k)}{\varrho}}^2\rd x\nonumber\\
	&+\int\limits_{\dom}\frac{\varrho}{r^\e}\left(\frac{\nbf^\e}{r^\e}-\frac{\mbf}{\varrho}\right)\cdot\left(\mathcal{K}_1^\e+\mathcal{K}_2^\e\right)\,\rd x.
	\end{align}
	We take $\varphi=\zeta(x-y)$ in one-sided condition \eqref{ineq:one_sided}, then we have
	\begin{equation}\label{mollified:condition}
	\nabla_x\left(\frac{\nbf^\e}{r^\e}\right):(\xi\otimes\xi)\geq -\chi(t)\abs{\xi}^2+\mathcal{K}_3^\e:(\xi\otimes\xi),
	\end{equation}
	where $\mathcal{K}_3^\e=\nabla_x\left(\frac{\nbf^\e}{r^\e}\right)-\nabla_x\left(\frac{\nbf}{r}\right)^\e$. Note that $p(\varrho)-(\varrho-r^\e) p^\p(r^\e)-p(r^\e)\geq c_0\abs{r^\e-\varrho}^2$ for some $c_0>0$ since $\varrho,r\geq \underline{r}>0$. Next, we take $\xi=\sqrt{\varrho}\left(\frac{\nbf^\e}{r^\e}-\frac{\mbf}{\varrho}\right)$ and $\xi=e_i$ in \eqref{mollified:condition} respectively to obtain
	\begin{align}
	&-\int\limits_{\dom}\varrho\left(\frac{\nbf^\e}{r^\e}-\frac{\mbf}{\varrho}\right)\otimes\left(\frac{\nbf^\e}{r^\e}-\frac{\mbf}{\varrho}\right):\nabla_x\frac{\nbf^\e}{r^\e}\,\rd x-\int\limits_{\dom}\left(p(\varrho)-(\varrho-r^\e) p^\p(r^\e)-p(r^\e)\right)\dv_x\frac{\nbf^\e}{r^\e}\,\rd x\nonumber\\
	&\leq \int\limits_{\dom}\chi(t)\varrho\abs{\frac{\nbf^\e}{r^\e}-\frac{\mbf}{\varrho}}^2+c_0\chi(s)\abs{r^\e-\varrho}^2\,\rd x+\int\limits_{\dom}C_1\abs{\mathcal{K}_3^\e}\,\rd x.
	\end{align}
	From \eqref{cal-R}, we get
	\begin{align}\label{cal-R-1}
	\int\limits_{t_1}^{t_2}\bar{\mathcal{R}}(\varrho,\mbf|r^\e,\nbf^\e)(s)\rd s&\leq 	\int\limits_{t_1}^{t_2}\int\limits_{\dom}\chi(s)\varrho\abs{\frac{\nbf^\e}{r^\e}-\frac{\mbf}{\varrho}}^2+c_0\chi(s)\abs{r^\e-\varrho}^2\,\rd x\rd s\nonumber\\
	&+	\int\limits_{t_1}^{t_2}\int\limits_{\dom}C_1\left(\abs{\mathcal{K}_1^\e}+\abs{\mathcal{K}_2^\e}+\abs{\mathcal{K}_3^\e}\right)\,\rd x\rd s\nonumber\\ 
	&	+	\int\limits_{t_1}^{t_2}\int\limits_{\dom}\sum\limits_{k\geq1}\frac{\varrho}{2}\abs{\frac{\mathbb{G}(r,\nbf)^\e(e_k)}{r^\e}-\frac{\mathbb{G}(\varrho,\mbf)(e_k)}{\varrho}}^2\rd x\rd s.
	\end{align}
	By using coercivity of relative entropy we obtain 
	\begin{align}\label{cal-R-2}
	\int\limits_{t_1}^{t_2}\bar{\mathcal{R}}(\varrho,\mbf|r^\e,\nbf^\e)(s)\rd s&\leq 	\int\limits_{t_1}^{t_2}C_2\chi(s)\mathcal{E}\left(\varrho,\mbf|r^\e,\nbf^\e\right)\,\rd s
	+	\int\limits_{t_1}^{t_2}\int\limits_{\dom}C_1\left(\abs{\mathcal{K}_1^\e}+\abs{\mathcal{K}_2^\e}+\abs{\mathcal{K}_3^\e}\right)\,\rd x\rd s\nonumber\\ 
		&	+	\int\limits_{t_1}^{t_2}\int\limits_{\dom}\sum\limits_{k\geq1}\frac{\varrho}{2}\abs{\frac{\mathbb{G}(r,\nbf)^\e(e_k)}{r^\e}-\frac{\mathbb{G}(\varrho,\mbf)(e_k)}{\varrho}}^2\rd x\rd s.
	\end{align}
	Then we apply Lemma \ref{lemma:commutator1} in \eqref{cal-R-2} to get
	\begin{align}\label{cal-R-3}
	\int\limits_{t_1}^{t_2}\bar{\mathcal{R}}(\varrho,\mbf|r^\e,\nbf^\e)(s)\rd s&\leq 	\int\limits_{t_1}^{t_2}C_2\chi(s)\mathcal{E}\left(\varrho,\mbf|r^\e,\nbf^\e\right)\,\rd s
	+C_1\abs{(r,\nbf)}_{B^{\al,\f}_{q}([t_1,t_2]\times\dom)}\e^{2\al-1}\nonumber\\
			&	+	\int\limits_{t_1}^{t_2}\int\limits_{\dom}\sum\limits_{k\geq1}\frac{\varrho}{2}\abs{\frac{\mathbb{G}(r,\nbf)^\e(e_k)}{r^\e}-\frac{\mathbb{G}(\varrho,\mbf)(e_k)}{\varrho}}^2\rd x\rd s.
	\end{align}
	By using the above estimate \eqref{cal-R-3}, Lemmas \ref{lemma:limit_pass_diff} and \ref{lemma:limit_pass_drift1} we pass to the limit in \eqref{cal-rel1} to obtain the following,
	\begin{align}
	\mathcal{E}(\varrho,\mbf|r,\nbf)(t_2)&\leq \mathcal{E}(\varrho,\mbf|r,\nbf)(t_1)+\int\limits_{t_1}^{t_2} M_{rem}\rd W_s+ 	\int\limits_{t_1}^{t_2}C_2\chi(s)\mathcal{E}\left(\varrho,\mbf|r,\nbf\right)\,\rd s\nonumber\\
			&	+	\int\limits_{t_1}^{t_2}\int\limits_{\dom}\sum\limits_{k\geq1}\frac{\varrho}{2}\abs{\frac{\mathbb{G}(r,\nbf)(e_k)}{r}-\frac{\mathbb{G}(\varrho,\mbf)(e_k)}{\varrho}}^2\rd x\rd s.\label{cal-weak-strong}
	\end{align}
	Since $r,\varrho\geq \underline{r}$ for $0\leq t<\tau$, by using \eqref{con-G} we estimate the last term on RHS of \eqref{cal-weak-strong} to get the following
	\begin{align}
	\mathcal{E}(\varrho,\mbf|r,\nbf)(t_2)&\leq \mathcal{E}(\varrho,\mbf|r,\nbf)(t_1)+\int\limits_{t_1}^{t_2} M_{rem}\rd W_s+ 	\int\limits_{t_1}^{t_2}C_2\chi(s)\mathcal{E}\left(\varrho,\mbf|r,\nbf\right)\,\rd s.
	\end{align}
	Now we take average integral over $\Omega$ on both side of \eqref{cal-rel1} and get
	\begin{align}
	\E\left(\mathcal{E}(\varrho,\mbf|r,\nbf)(t_2)\right)\leq \E\left(\mathcal{E}(\varrho,\mbf|r,\nbf)(t_1)\right)
	+\int\limits_{t_1}^{t_2}C_2\chi(t)\E\left(\mathcal{E}\left(\varrho,\mbf|r,\nbf\right)\right)\,\rd s.
	\end{align}
	We pass to the limit as $t_1\rr0$ and use Gr\"{o}nwall's inequality to conclude Theorem \ref{theorem_uniquess_I}. 
\end{proof}

\section{Energy balance}\label{sec:energy-balance}
In this section we establish energy balance equation for the system \eqref{SCE1}--\eqref{SCE2}. This is done by mollifying the system in space variable and then passing to the limit by commutator estimate. We use some of the formulas obtained in the proof of Proposition \ref{prop:rel_energy}.
\begin{proof}[Proof of Theorem \ref{theorem:energy_balance}:]
	
	Since $(\varrho,\mbf)$ is a weak solution to \eqref{SCE1}--\eqref{SCE2}, it does not have pointwise sense. In order to that we mollify the system \eqref{SCE1}--\eqref{SCE2} in space variable by standard mollifiers, and obtain
	\begin{align}
	\rd \varrho^\e&=-\dv_x\mbf^\e\rd s,\label{eqn:mollified_rho1}\\
	\rd\mbf^\e&=\left(-\dv_x\left(\frac{\mbf\otimes\mbf}{\varrho}\right)^\e-\nabla_xp(\varrho)^\e\right)\rd s+\G(\varrho,\mbf)^\e\rd W_s.
	\end{align}
	After a modification we get
	\begin{equation}
		\rd\mbf^\e=-\dv_x\left(\frac{\mbf^\e\otimes\mbf^\e}{\varrho^\e}\right)\rd s-\nabla_xp(\varrho^\e)\rd s+\mathcal{T}_{rem}^\e\rd s+\mathbb{G}(\varrho,\mbf)^\e\rd W_s,
	\end{equation}
	where $\mathcal{T}_{rem}^\e$ is defined as 
	\begin{equation}
	\mathcal{T}_{rem}^\e:=\dv_x\left(\frac{\mbf^\e\otimes\mbf^\e}{\varrho^\e}\right)-\dv_x\left(\frac{\mbf\otimes\mbf}{\varrho}\right)^\e+\nabla_xp(\varrho^\e)-\nabla_xp(\varrho)^\e.
	\end{equation}
	From \eqref{rel1} with $r=\varrho^\e$, $\nbf=\mbf^\e$ and
	\begin{align*}
	&F_s=-\dv_x\mbf^\e,\\
	& A_s=0,\\
	&G_s=-\dv_x\left(\frac{\mbf^\e\otimes\mbf^\e}{\varrho^\e}\right)-\nabla_xp(\varrho^\e)+\mathcal{T}_{rem}^\e,\\
	& B_s=\mathbb{G}(\varrho,\mbf)^\e,
	\end{align*}
	we get the following
%
	\begin{align}
	\rd\left(\frac{\abs{\mbf^\e}^2}{2\varrho^\e}\right)&=\frac{\dv_x\mbf^\e}{2(\varrho^\e)^2}\abs{\mbf^\e}^2\rd s+\left[-\dv_x\left(\frac{\mbf^\e\otimes\mbf^\e}{\varrho^\e}\right)-\nabla_xp(\varrho^\e)\right]\cdot\frac{\mbf^\e}{\varrho^\e}\rd s\nonumber\\
	&+\frac{1}{2}\sum\limits_{k\geq1}\frac{\abs{\G(\varrho,\mbf)^\e(e_k)}^2}{\varrho^\e}\rd s+\G(\varrho,\mbf)^\e\cdot\frac{\mbf^\e}{\varrho^\e}\rd W_s+\mathcal{T}_{rem}^\e\cdot\frac{\mbf^\e}{\varrho^\e}\rd s.
	\end{align}
		By expanding the second term on RHS, we have
	\begin{align}
	\rd\left(\frac{\abs{\mbf^\e}^2}{2\varrho^\e}\right)&=\frac{\abs{\mbf^\e}^2}{2(\varrho^\e)^2}\dv_x\mbf^\e\rd s
	-\dv_x\mbf^\e\frac{\abs{\mbf^\e}^2}{(\varrho^\e)^2}\rd s-\mbf^\e\cdot\nabla_x\left(\frac{\mbf^\e}{\varrho^\e}\right)\cdot\frac{\mbf^\e}{\varrho^\e}\rd s\nonumber\\
	&-\frac{p^{\p}(\varrho^\e)}{\varrho^\e}\mbf^\e\cdot\nabla_x\varrho^\e\rd s+\mathcal{T}_{rem}^\e\cdot\frac{\mbf^\e}{\varrho^\e}\rd s\nonumber\\
	&+\frac{1}{2}\sum\limits_{k\geq1}\frac{\abs{\G(\varrho,\mbf)^\e(e_k)}^2}{\varrho^\e}\rd s+\G(\varrho,\mbf)^\e\cdot\frac{\mbf^\e}{\varrho^\e}\rd W_s.
	\end{align}
Using \eqref{derivative:P} we simplify further to get the following 
	\begin{align}
	\rd\left(\frac{\abs{\mbf^\e}^2}{2\varrho^\e}\right)&=
	-\dv_x\mbf^\e\frac{\abs{\mbf^\e}^2}{2(\varrho^\e)^2}\rd s-\mbf^\e\cdot\nabla_x\left(\frac{\abs{\mbf^\e}^2}{2(\varrho^\e)^2}\right)\rd s-P^{\p\p}(\varrho^\e)\mbf^\e\cdot\nabla_x\varrho^\e\rd s\nonumber\\
	&+\frac{1}{2}\sum\limits_{k\geq1}\frac{\abs{\G(\varrho,\mbf)^\e(e_k)}^2}{\varrho^\e}\rd s+\G(\varrho,\mbf)^\e\cdot\frac{\mbf^\e}{\varrho^\e}\rd W_s+\mathcal{T}_{rem}^\e\cdot\frac{\mbf^\e}{\varrho^\e}\rd s.\label{cal:energy_kinetic}
	\end{align}
	By using It\^{o}'s formula for the function $r\mapsto P(r)$ with respect to \eqref{eqn:mollified_rho1}, we obtain
	\begin{equation}\label{cal:energy_internal}
	\rd P(\varrho^\e)=-P^{\p}(\varrho^\e)\dv_x\mbf^\e\rd s.
	\end{equation}
	Combining \eqref{cal:energy_kinetic} and \eqref{cal:energy_internal} we get
	\begin{align}
	\rd\left(\frac{\abs{\mbf^\e}^2}{2\varrho^\e}+P(\varrho^\e)\right)&=
	-\dv_x\left({\mbf^\e}\left(\frac{\abs{\mbf^\e}^2}{2(\varrho^\e)^2}+P^{\p}(\varrho^\e)\right)\right)\rd s+\mathcal{T}_{rem}^\e\cdot\frac{\mbf^\e}{\varrho^\e}\rd s\nonumber\\
	&+\frac{1}{2}\sum\limits_{k\geq1}\frac{\abs{\G(\varrho,\mbf)^\e(e_k)}^2}{\varrho^\e}\rd s+\G(\varrho,\mbf)^\e\cdot\frac{\mbf^\e}{\varrho^\e}\rd W_s.
	\end{align}
	Applying $\eta\in C_c^{\f}(\R_+)$ and $\varphi\in C_c^{\f}(\dom)$ we have
	\begin{align}
	-\int\limits_{0}^{T}\pa_t\eta \int\limits_{\dom}\left(\frac{\abs{\mbf^\e}^2}{2\varrho^\e}+P(\varrho^\e)\right)\varphi\,\rd x\rd s
	&=-\int\limits_{0}^{T}\eta \int\limits_{\dom}\dv_x\left({\mbf^\e}\left(\frac{\abs{\mbf^\e}^2}{2(\varrho^\e)^2}+P^{\p}(\varrho^\e)\right)\right)\varphi\,\rd x\rd s\nonumber\\
	&+\int\limits_{0}^{T}\eta \int\limits_{\dom}\frac{1}{2}\sum\limits_{k\geq1}\frac{\abs{\G(\varrho,\mbf)^\e(e_k)}^2}{\varrho^\e}\varphi\rd x\rd s\nonumber\\
	&+\int\limits_{0}^{T}\eta \int\limits_{\dom}\G(\varrho,\mbf)^\e\cdot\frac{\mbf^\e}{\varrho^\e}\varphi\rd x\rd W_s
		+\int\limits_{0}^{T}\eta \int\limits_{\dom}\mathcal{T}_{rem}^\e\cdot\frac{\mbf^\e}{\varrho^\e}\varphi\rd x\rd s.
	\end{align}
	Using integration by parts we get
	\begin{align}
	-\int\limits_{0}^{T}\pa_t\eta \int\limits_{\dom}\left(\frac{\abs{\mbf^\e}^2}{2\varrho^\e}+P(\varrho^\e)\right)\varphi\,\rd x\rd s
	&=\int\limits_{0}^{T}\eta \int\limits_{\dom}\left({\mbf^\e}\left(\frac{\abs{\mbf^\e}^2}{2(\varrho^\e)^2}+P^{\p}(\varrho^\e)\right)\right)\cdot\nabla_x\varphi\,\rd x\rd s\nonumber\\
	&+\int\limits_{0}^{T}\eta \int\limits_{\dom}\frac{1}{2}\sum\limits_{k\geq1}\frac{\abs{\G(\varrho,\mbf)^\e(e_k)}^2}{\varrho^\e}\varphi\rd x\rd s\nonumber\\
	&+\int\limits_{0}^{T}\eta \int\limits_{\dom}\G(\varrho,\mbf)^\e\cdot\frac{\mbf^\e}{\varrho^\e}\varphi\rd x\rd W_s
	+\int\limits_{0}^{T}\eta \int\limits_{\dom}\mathcal{T}_{rem}^\e\cdot\frac{\mbf^\e}{\varrho^\e}\varphi\rd x\rd s.\label{energy-cal1}
	\end{align}
	By using Lemma \ref{lemma:commutator} we get
	\begin{equation}\label{energy-cal2}
		\int\limits_{0}^{T}\eta \int\limits_{\dom}\mathcal{T}_{rem}^\e\cdot\frac{\mbf^\e}{\varrho^\e}\varphi\rd x\rd s\rr0\mbox{ as }\e\rr0\mbox{ in }\mathbb{P}-\mbox{a.s.}
	\end{equation}
	We apply Lemma \ref{lemma:limit_pass_diff} to obtain
	\begin{align}\label{energy-cal3}
	\int\limits_{0}^{T}\eta(s) \int\limits_{\dom}\G(\varrho,\mbf)^\e\cdot\frac{\mbf^\e}{\varrho^\e}\varphi\rd x\rd W_s\rr	\int\limits_{0}^{T}\eta(s) \int\limits_{\dom}\G(\varrho,\mbf)\cdot\frac{\mbf}{\varrho}\varphi\rd x\rd W_s\mbox{ as }\e\rr0\,\mbox{in }\mathbb{P}-\mbox{a.s.}
	\end{align}
	By Lemma \ref{lemma:limit_pass_drift1} we have the following
		\begin{align}
	-\int\limits_{0}^{T}\pa_t\eta(s) \int\limits_{\dom}\left(\frac{\abs{\mbf^\e}^2}{2\varrho^\e}+P(\varrho^\e)\right)\varphi\,\rd x\rd s
	&\rr	-\int\limits_{0}^{T}\pa_t\eta(s) \int\limits_{\dom}\left(\frac{\abs{\mbf}^2}{2\varrho}+P(\varrho)\right)\varphi\,\rd x\rd s,\label{energy-cal4}\\
\int\limits_{0}^{T}\eta(s) \int\limits_{\dom}\left({\mbf^\e}\left(\frac{\abs{\mbf^\e}^2}{2(\varrho^\e)^2}+P^{\p}(\varrho^\e)\right)\right)\cdot\nabla_x\varphi\,\rd x\rd s&\rr\int\limits_{0}^{T}\eta(s) \int\limits_{\dom}\left({\mbf}\left(\frac{\abs{\mbf}^2}{2\varrho^2}+P^{\p}(\varrho)\right)\right)\cdot\nabla_x\varphi\,\rd x\rd s,\label{energy-cal5}\\
	\int\limits_{0}^{T}\eta(s) \int\limits_{\dom}\frac{1}{2}\sum\limits_{k\geq1}\frac{\abs{\G(\varrho,\mbf)^\e(e_k)}^2}{\varrho^\e}\varphi\rd x\rd s&\rr \int\limits_{0}^{T}\eta(s) \int\limits_{\dom}\frac{1}{2}\sum\limits_{k\geq1}\frac{\abs{\G(\varrho,\mbf)(e_k)}^2}{\varrho}\varphi\rd x\rd s,\label{energy-cal6}
	\end{align}
	as $\e\rr0$ in $\mathbb{P}$--a.s. Clubbing \eqref{energy-cal1}--\eqref{energy-cal6} we conclude Theorem \ref{theorem:energy_balance}.
\end{proof}

\section*{Appendix}
\begin{proof}[Proof of Lemma \ref{lemma:limit_pass_diff}:]
	We first show the convergence in mean square. By triangle inequality, we have
	\begin{align*}
	\E\left(\abs{\int\limits_{0}^T\int\limits_{\dom} Y_s^\e D^\e_s\rd x\rd W_s-\int\limits_{0}^T\int\limits_{\dom} Y_sD_s\rd x\rd W_s
	}^2\right)
	&\leq 2\E\left(\abs{\int\limits_{0}^T\int\limits_{\dom} Y_s^\e\left(D^\e_s-D_s\right)\rd x\rd W_s}^2\right)\\
	&+ 2\E\left(\abs{\int\limits_{0}^T\int\limits_{\dom}\left(Y_s^\e-Y_s\right)D_s\rd x \rd W_s}^2\right).
	\end{align*}
	By using B\"urkhholder-Davis-Gundy inequality \cite{bfh-book} and H\"older inequality we get
	\begin{align}
	&	\E\left(\abs{\int\limits_{0}^T\int\limits_{\dom} Y_s^\e D^\e_s\rd x\rd W_s-\int\limits_{0}^T\int\limits_{\dom} Y_sD_s\rd x\rd W_s
	}^2\right)
	\nonumber\\
	&\leq 2C(2)\E\left(\int\limits_{0}^T\sum\limits_{k\geq1}\norm{ Y_s^\e}^2_{L^q(\dom)}\norm{D^\e_s(e_k)-D_s(e_k)}^2_{L^{\frac{q}{q-1}}(\dom)}\,\rd s\right)\nonumber\\
	&+ 2C(2)\E\left(\int\limits_{0}^T\sum\limits_{k\geq1}\norm{  Y_s^\e-Y_s}^2_{L^q(\dom)}\norm{D_s(e_k)}^2_{L^{\frac{q}{q-1}}(\dom)}\,\rd s\right).
	\end{align}
	By using H\"older inequality we observe that
	\begin{align}
	&\E\left(\int\limits_{0}^T\sum\limits_{k\geq1}\norm{ Y_s^\e}^2_{L^q(\dom)}\norm{D^\e_s(e_k)-D_s(e_k)}^2_{L^{\frac{q}{q-1}}(\dom)}\,\rd s\right)\nonumber\\
	&\leq  C(T)\left(\E\left(\sup\limits_{0\leq t\leq T}\norm{ Y_s^\e}^2_{L^q(\dom)}\right)^2\right)^{1/2}\left(\E\left(\int\limits_{0}^T\sum\limits_{k\geq1}\norm{D^\e_s(e_k)-D_s(e_k)}^2_{L^{\frac{q}{q-1}}(\dom)}\rd s\right)^2\right)^{1/2}.\label{cal-lemma1}
	\end{align}
	Recall the integrability assumption \eqref{integrability:D} on $D_s$. Now we note that for a.e. $t\in[0,T]$ and $\omega\in\Omega$ we have 
	\begin{equation}\label{bound:D}
	\sum\limits_{k\geq1}\norm{D^\e_s(e_k)-D_s(e_k)}^2_{L^{\frac{q}{q-1}}(\dom)}(\omega,t)\leq 4\sum\limits_{k\geq1}\norm{D_s}_{L^{\frac{q}{q-1}}(\dom)}^2<\f.
	\end{equation}
	We also observe that 
	\begin{equation}\label{cal-lemma2}
	\E\left(\int\limits_{0}^T\sum\limits_{k\geq1}\norm{D^\e_s(e_k)-D_s^\e(e_k)}^2_{L^{\frac{q}{q-1}}(\dom)}\rd s\right)^2\leq 4\E\left(\int\limits_{0}^T\sum\limits_{k\geq1}\norm{D_s(e_k)}^2_{L^{\frac{q}{q-1}}(\dom)}\rd s\right)^2<\f.
	\end{equation}
	Because of \eqref{bound:D}, for a fixed $t\in[0,T]$, $\omega\in\Omega$ and $\de>0$, we have
	\begin{equation*}
	\sum\limits_{k\geq1}\norm{D^\e_s(e_k)-D^\e_s(e_k)}^2_{L^{\frac{q}{q-1}}(\dom)}(\omega,t)\leq \sum\limits_{k=1}^{N_\de(\omega,t)}\norm{D^\e_s(e_k)-D^\e_s(e_k)}^2_{L^{\frac{q}{q-1}}(\dom)}(\omega,t)+\de,
	\end{equation*}
	for some $N_\de(\omega,t)\geq1$. Note that $\norm{D^\e_s(e_k)-D_s(e_k)}^2_{L^{\frac{q}{q-1}}(\dom)}(\omega,t)\rr0$ as $\e\rr0$ for a.e. $t\in[0,T]$ and $\omega\in\Omega$. Therefore, we obtain
	\begin{equation*}
	\limsup\limits_{\e\rr0+}\sum\limits_{k\geq1}\norm{D^\e_s(e_k)-D_s(e_k)}^2_{L^{\frac{q}{q-1}}(\dom)}(\omega,t)\leq\de\mbox{ for a.e. } t\in[0,T]\mbox{ and }\omega\in\Omega.
	\end{equation*}
	Since $\de>0$ is arbitrary, we get
	\begin{equation*}
	\lim\limits_{\e\rr0+}\sum\limits_{k\geq1}\norm{D^\e_s(e_k)-D^\e_s(e_k)}^2_{L^{\frac{q}{q-1}}(\dom)}(\omega,t)=0.
	\end{equation*} 
	By Lebesgue Dominated Convergence Theorem we get
	\begin{equation}\label{limit:D}
	\E\left(\int\limits_{0}^T\sum\limits_{k\geq1}\norm{D^\e_s(e_k)-D^\e_s(e_k)}^2_{L^{\frac{q}{q-1}}(\dom)}\right)^2\rr0\mbox{ as }\e\rr0.
	\end{equation}
	By using \eqref{integrability:Y}, \eqref{limit:D}, \eqref{cal-lemma2} and Lebesgue Dominated Convergence Theorem we pass to the limit in \eqref{cal-lemma1} as $\e\rr0$ to get 
	\begin{equation*}
	\E\left(\int\limits_{0}^T\sum\limits_{k\geq1}\norm{ Y_s^\e}^2_{L^q(\dom)}\norm{D^\e_s(e_k)-D_s(e_k)}^2_{L^{\frac{q}{q-1}}(\dom)}\rd s\right)\rr0\mbox{ as }\e\rr0.
	\end{equation*}
	Since we have
	\begin{equation*}
	\E\left(\sup\limits_{t\in[0,T]}\norm{Y_s}_{L^q(\dom)}\right)^2<\f\mbox{ and }\E\left(\int\limits_{0}^T\sum\limits_{k\geq1}\norm{D_s(e_k)}^2_{L^{\frac{q}{q-1}}(\dom)}\right)^2<\f,
	\end{equation*}
	we obtain
	\begin{equation*}
	\E\left(\int\limits_{0}^T\sum\limits_{k\geq1}\norm{ Y^\e_s-Y_s}^2_{L^q(\dom)}\norm{D_s (e_k)}^2_{L^{\frac{q}{q-1}}(\dom)}\right)\,dt\rr0\mbox{ as }\e\rr0.
	\end{equation*}
	Therefore, we get
	\begin{equation*}
		\E\left(\abs{\int\limits_{0}^T\int\limits_{\dom} Y_s^\e D^\e_s\rd x\rd W_s-\int\limits_{0}^T\int\limits_{\dom} Y_sD_s\rd x\rd W_s
	}^2\right)
\rr0\mbox{ as }\e\rr0.
	\end{equation*}
	Hence, up to a subsequence we conclude Lemma \ref{lemma:limit_pass_diff}.
\end{proof}

\section*{Acknowledgements} SSG thanks Martina Hofmanov\'a for enlightening discussion on this topic during her visit to TIFR-CAM. Authors would like to acknowledge the support of the Department of Atomic Energy, Government of India, under project no. 12-R\&D-TFR-5.01-0520. SSG also thanks the Inspire faculty-research grant DST/INSPIRE/04/2016/000237.

\end{document}